\documentclass[12pt]{amsart}

\usepackage{latexsym,amssymb,amscd,amsmath,mathrsfs,bbm,stmaryrd}
\usepackage[notref,notcite]{showkeys}
\usepackage{hyperref}
\usepackage[all]{xy}

\newtheorem{innercustomgeneric}{\customgenericname}
\providecommand{\customgenericname}{}
\newcommand{\newctheorem}[2]{%
  \newenvironment{#1}[1]
  {%
   \renewcommand\customgenericname{#2}%
   \renewcommand\theinnercustomgeneric{##1}%
   \innercustomgeneric
  }
  {\endinnercustomgeneric}
}

\newctheorem{cthm}{Theorem}
\newctheorem{clem}{Lemma}
\newctheorem{cprp}{Proposition}
\newctheorem{ccor}{Corollary}

\newtheorem{thm}{Theorem}[section]
\newtheorem{lem}[thm]{Lemma}
\newtheorem{prp}[thm]{Proposition}
\newtheorem{cor}[thm]{Corollary}
\newtheorem*{thmd}{Theorems \& Definitions}

\theoremstyle{definition}
\newtheorem{dfn}[thm]{Definition}

\theoremstyle{definition}
\newtheorem*{exm}{Examples}

\theoremstyle{remark}
\newtheorem*{rmk}{Remark}

\numberwithin{equation}{section}

\def\Ga{\Gamma}       
     \def\Si{\Sigma}
       \def\th{\theta}
\def\al{\alpha}       
\def\be{\beta}        
       
\def\la{\lambda}      
\def\vi{\varphi}      \def\io{\iota}
\def\si{\sigma}

\def\bV{\mathbf{V}}			\def\bW{\mathbf{W}}
\def\bF{\mathbf{F}}				
\def\bS{\mathbf{S}}			
\def\bB{\mathbf{B}}
\def\bal{{\boldsymbol\alpha}}
\def\aK{\mathbbm k}
\def\gR{\mathfrak{r}}		\def\gM{\mathfrak{m}}
\def\cS{\mathscr S}			\def\cP{\mathscr P}
\def\rD{\mathrm{D}}

\def\cA{\mathscr A}

	\def\cM{\mathscr M}

\def\cI{\mathscr I}

\def\cM{\mathscr M} 

\def\ovi{\bar{\vi}}		
				
\def\oal{\bar{\al}}		
			\def\tU{\tilde{U}}
\def\tH{\tilde{H}}

\def\sb{\subset}         
\def\spe{\supseteq}      \def\sbe{\subseteq}
\def\*{\otimes}          
\def\bop{\bigoplus}	\def\+{\oplus}

		\def\ito{\stackrel\sim\to}
\def\mt{\mbox{-}}		
\def\xarr{\xrightarrow}	
\def\vv{^\vee}			\def\mps{\mapsto}
\def\8{\infty}			
\def\bb{^\flat}			\def\={{\setminus}}

\def\lst#1#2{ #1_1 , #1_2 , \dots , #1_{#2} }
\def\mtr#1{\begin{pmatrix}#1\end{pmatrix}}

\def\iff{if and only if }
\def\wrp{with respect to}
\def\AR{Auslander-Reiten}
\def\as{almost split}
\def\ass{almost split sequence}
\def\id{\mathrm{Id}}		

\def\soc{\mathop\mathrm{soc}\nolimits}
\def\Mdd{\mt\mathrm{Mod}}
\def\pro{\mt\mathrm{proj}}
\def\tr{\mathop\mathrm{Tr}\nolimits}
\def\ver{\mathop\mathrm{ver}\nolimits}
\def\hom{\mathop\mathrm{Hom}\nolimits}

\def\ker{\mathop\mathrm{ker}\nolimits}
\def\im{\mathop\mathrm{im}\nolimits}
\def\cok{\mathop\mathrm{coker}\nolimits}
\def\rep{\mathop\mathrm{Rep}\nolimits}
\def\rad{\mathop\mathrm{rad}\nolimits}
\def\mdd{\mbox{-}\mathrm{mod}}

\def\Mdi{\mbox{-}{\overline{\mathrm{Mod}}}}
\def\Mdp{\mbox{-}{\underline{\mathrm{Mod}}}}

\def\ob{\mathop\mathrm{Ob}\nolimits}
\def\End{\mathop\mathrm{End}\nolimits}

\def\Rad{\mathop\mathrm{Rad}\nolimits}
\def\irr{\mathop\mathrm{Irr}\nolimits}
\def\ars{\mathop\mathrm{AR}\nolimits}
\def\ind{\mathop\mathrm{ind}\nolimits}

\begin{document}

 \title[Algebras with radical square zero]{On representations of algebras with radical square zero}
 \author{Yuriy A. Drozd}
 \address{Institute of Mathematics of the National Academy of Sciences of Ukraine, Tereschernkivska
 str. 3, Kyiv 01026, Ukraine}
 \email{y.a.drozd@gmail.com,\,drozd@imath.kiev.ua}
 

  \begin{abstract}
  We consider a new correspondence between representations of algebras with radical square zero and 
  representations of species. We show that the stable category of representations of such algebra embeds
  into the representation category of the corresponding species and show how one reconstruct the \AR\
  quiver of the algebra from the \AR\ quiver of the species.
  \end{abstract}

\maketitle

 Algebras with radical square zero were studied by a lot of investigators, as they are a convenient testing area for new
 methods and approaches of the representation theory. The first results about their representations were obtained by
 Yoshii \cite{Yo1,Yo2} who found conditions for such an algebra to be representation finite. Kruglyak \cite{Krug} related
 representations of such algebra to representations of posets (actually, to representations of quivers, but not using this
 term) and rereceived in this way the results of Yoshii with a small refinement. This approach was further used in a series
 of papers \cite{BL,Bek,Yang} to study derived categories of such algebras. Auslander and Reiten \cite{AR} established 
 a stable equivalence of an algebra with radical square zero with a certain species (or realization of a valued graph in 
 the sense of \cite{DR}). In this paper we consider another variant of correspondence between algebras with radical
 square zero and species. We prove that the stable category of such algebra embeds into the category of representations
 of the corresponding species and show how the \AR\ quiver of the algebra can be reconstructed from the
 \AR\ quiver of the corresponding species.

 \section{Species}
 \label{s1}
 
 Recall the notions of species and their representations.
 
  \begin{dfn}\label{d11} 
  \begin{enumerate}
  \item    A \emph{species} $\Si$ consists of:
  \begin{itemize}
   \item  A set of \emph{vertices} $\ver\Si$.
   \item  For each vertex $i$ a skew field $F_i$.
   \item  For each pair of vertices $(i,j)$ an $F_i\mt F_j$-bimodule (left over $F_i$, right over $F_j$) $B(i,j)$.
  \end{itemize} 
  \item   A \emph{$\Si$-module} (or a \emph{representation of the species $\Si$}) $\bV$ is a set $(V_i,\vi_{ij})$, where each 
  $V_i$ is a vector space over $F_i$ and $\vi_{ij}$ is an $F_i$-morphism $B(i,j)\*_{F_j}V_j\to V_i$.
  \item  A \emph{morphism} of representations $\bal:\bV\to\bW$, where $\bW=(W_i,\psi_{ij})$ is a set of $F_i$-morphisms
  $\al_i:V_i\to W_i$ such that for every pair $(i,j)$ the diagram
  \[
   \xymatrix{ B(i,j)\*_{F_j}V_j \ar[rr]^{\vi_{ij}} \ar[d]_{1\*\al_j} && V_i \ar[d]^{ \al_i} \\
    					B(i,j)\*_{F_j}W_j \ar[rr]^{\psi_{ij}}  && W_i } 
 \]
 is commutative.
  \end{enumerate}
   \end{dfn} 
   
   If all $F_i$ conincide with a fixed field $\aK$ and $\la x=x\la$ for every $x\in B(i,j)$ and $\la\in\aK$, the species is
   completely defined if we fix $d_{ij}=\dim_\aK B(i,j)$. Then they usually identify the species $\Si$ with the quiver
   (oriented graph) with $\ver\Si$ as the set of vertices and $d_{ij}$ arrows with the source $j$ and target $i$. 
   Representations of $\Si$ are the same as the representations of this quiver in the sense of \cite{gab1}. In general
   case, species and their representations are \emph{realizations of valued graphs} and their representations in the
   sense of \cite{Dlab,DR}.
   
 The $\Si$-modules form an abelian category $\Si\Mdd$ (or $\rep\Si$). If $\ver\Si$ is finite, it can be identfied 
 with the category of modules over the \emph{tensor algebra} $T_\bF(\bB)=\bop_{k=0}^\8\bB^{\*k}$, where 
 $\bF=\prod_iF_i$ and $\bB=\bop_{i,j}B_{ij}$ with the natural action of $\bF$, but we prefer the language of species.
 
 Let $\cA$ be a \emph{local category}, that is the following conditions hold:
 \begin{enumerate}
 \item  $\cA$ is \emph{preadditive}, that is all sets of morphisms $\cA(X,Y)$ are abelian groups and the multiplication
 is distributive \wrp\ addition.
 \item  Every object is a finite direct sum of \emph{local objects}, that is objects with local endomorphism rings.
 \end{enumerate}
 It is known that such decomposition is unique up to isomorhism. Moreover, every idempotent in a local category
 \emph{splits}, that is arise from a direct decomposition (see, for instance \cite[Ch.1,\S3]{Bass}).
 If a local category has finitely many objects, it can be considered as the category $A\pro$ of finite projective 
 modules over the semiperfect ring $A=\sum_{X,Y\in\ob\cA}\cA(X,Y)$.
 
 Let $X,Y$ be objects of a local category $\cA$. Fix there decompositions into direct sums of local objects
 and denote by $\Rad(X,Y)$ the set of morphisms from $\cA(X,Y)$ such that all their components \wrp\ these 
 decompositions are non-invertible (it does not depend on the chosen decompositions). We obtain an ideal in
 $\cA$ called the \emph{radical} $\Rad\cA$ of the category $\cA$. We can consider its	 powers $\Rad^k(\cA)$
 and $\Rad^\8(\cA)=\cap_{k=1}^\8\Rad^\8(\cA)$. Set $\irr(X,Y)=\Rad(X,Y)/\Rad^2(X,Y)$. We call elements of
 this set as well as their representatives in $\Rad(X,Y)$ the \emph{irreducible morphisms} from $X$ to $Y$.
 
  \begin{dfn}\label{d12} 
  Let $\cA$ be a local category. Its \emph{species} $\Si=\Si(\cA)$ is defined as follows:
  \begin{itemize}
  \item $\ver\Si=\ind\cA$, the set of representatives of isomorphism classes of local objects.
  \item $F_X=\cA(X,X)/\Rad(X,X)$.
  \item $B(X,Y)=\irr(Y,X)$.
  \end{itemize}
  \end{dfn}

  If $A$ is a semiperfect ring, they write $\Si(A)$ instead of $\Si(A\pro)$. It is called the \emph{Gabriel species}
  of $A$ (\emph{Gabriel quiver} if this species is identified with a quiver). Actually, it has been introduced in \cite{gab2}. 
  
  If $A$ is a finite algebra over a henselian noetherian commutative ring $R$ and $\cA$ is a full subcategory of the
  category $A\mdd$ of finite $A$-modules, closed under direct sums and direct summands, it is a local category.
  The species $\Si(\cA)$ is then called the \emph{\AR\ species} of the category $\cA$. If $\cA=A\mdd$, it is also
  called the \emph{\AR\ species} of the algebra $A$ and denoted by $\ars(A)$.

  We call a species $\Si$ \emph{bipartite} if $\ver\Si=\ver^+\sqcup\ver^-$ and $H(i,j)=0$ if $i\notin\ver^-$ or 
  $j\notin\ver^+$. For a vertex $i$ let $\bS_i$ be a $\Si$-module such that $\bS_i(i)=F_i$ and $\bS_i(j)=0$ if $j\ne i$.
  
   \begin{prp}\label{p13} 
   Let $\Si$ be a bipartite species. We denote by $\Si\Mdd\bb$ the full subcategory of $\Si\Mdd$ consisting of all
   modules without simple direct summands.
      \begin{enumerate}
      \item $\bS_i$ are all simple $\Si$-modules.
      \item $\bS_i$ is projective if $i\in \ver^+$ and injective if $i\in\ver^-$.
      \item $\Si\Mdd\bb\simeq\Si\Mdd/\cS$, where $\cS$ is the ideal of all morphisms that factor through semisimple
      modules. 
      \end{enumerate}   
    \end{prp}
    \begin{proof}
    [\rm The proof is evident]
    \end{proof}
    
     \begin{rmk}
	 If $\Si$ is not bipartite, there can be other simple $\Si$-modules. For instance, if $\Si$ is given by the quiver
     $1\leftrightarrows2$ and a field $\aK$, all modules $\bV$ with $V_1=V_2=\aK$, $\vi_{12}=\id$ and 
     $\vi_{21}=\la\cdot\id\ (\la\in\aK\=\{0\})$
     are simple. Moreover, if $f(t)\ne t$ is an irreducible polynomial over $\aK$, $\vi$ is a matrix with the characteristic
     polynomial $f(t)$, the module $\bV$ with $V_1=V_2=\aK^d$, where $d=\deg f(x)$, $\vi_{12}=\id$ and 
     $\vi_{21}=\vi$ is also simple.
          \end{rmk}
    
     Note that if $\Si$ is bipartite and $\ver\Si$ is finite, the tensor algebra $T_\bF(\bB)$ is isomorphic to the algebra 
     of triangular matrices
    \begin{equation}\label{eq0} 
        \aK[\Si] =  \mtr{\bF^+ & \bB \\ 0 & \bF^-},
    \end{equation}
     where $\bF^\pm=\prod_{i\in\ver^\pm}F_i$. It is a semiprimary ring with $\rad^2T_\bF(\bB)=0$, 
     which is left (right) artinian \iff $\dim_{F_i}B(i,j)<\8$ (respectively, $\dim_{F_j}B(i,j)<\8$) for all $i\in\ver^+,\,j\in\ver^-$.
  
 \section{Representations}
 \label{s2}
 
 In this section $A$ denotes a semiprimary ring with $\gR^2=0$, where $\gR=\rad A$. Such ring is \emph{perfect} 
 in the sense of \cite{fd}, hence for every $A$-module $M$ there is a \emph{minimal projective presentation}, 
 that is a morphism of projective modules $\vi:P^-\to P^+$ such that 
 $\im\vi\sbe\gR P^+,\ \ker\vi\sbe\gR P^-$ and $\cok\vi\simeq M$. Such presentation is unique up to isomorphism. 
 By $A\Mdp$ we denote the quotient $A\Mdd/\cP$, where $\cP$ is the ideal of morphisms that factor through 
 projective modules. 
 
 Let  $\lst Pn$ be
representatives of isomorphism classes of indecomposable projective $A$-modules. For a projective $A$-module 
$P$ we denote by $P(i)$ its \emph{$P_i$-part}, that is a direct summand of $P$ such that $P=P(i)\+Q$, where 
$P(i)\simeq P_i^k$ and $Q$ has no direct summands isomorphic to $P_i$. Then $P=\bop_{i=1}^nP(i)$. The species
$\Si(A)$ consists of vertices $\lst Pn$, skew fields $F_i=\End_AP_i/\rad\End_AP_i$ and bimodules
$B_{ij}=\hom_A(P_j,\gR P_i)$. 

 We also consider another species $\Ga=\Ga(A)$, sometimes called the \emph{double species} of the algebra $A$.
 It is bipartite with the vertices $\ver^\pm\Ga=\{i^\pm\mid 1\le i\le n\}$, the skewfields $F_{i^\pm}=F_i$ and the 
 bimodules $H_{i^+j^-}=\hom_A(P_j,\gR P_i)$. For a $\Ga$-module $\bV$ we write $V_i^\pm$ instead of $V_{i^\pm}$.
 As above, we denote by $\Ga\Mdd\bb$ the full subcategory of $\Ga\Mdd$ consisting of all modules $\bV$ without 
 simple direct summands. It is equivalent to the quotient $\Ga\Mdd/\cS$, where $\cS$ is the ideal of morphisms that factor
 through semisimple modules.
 
 In what follows we write $H_{ij}$ instead of $H_{i^+j^-}$.
 
  \begin{thm}\label{t21} 
  There is an equivalence $\Phi:A\Mdp\ito\Ga\Mdd\bb$.
  \end{thm}
  \begin{proof}
 For a module $M\in A\Mdd$, let $\vi:P^-\to P^+$ be its minimal projective presentation. As $\gR^2=0$, it can be identified
 with the induced map $\ovi:P^-/\gR P^-\to\gR P^+$. We define the corresponding $\Ga$-module $\bV=\Phi(M)$ as follows:
  \begin{itemize}
 \item  $V^+_i=\gR P^+(i)$,
 \item  $V^-_j=P^-(j)/\gR P^-(j)$,
 \item  morphisms $\vi_{i^+j^-}: H_{ij}\*_{F_j}V^-_j\to V^+_i$ are induced by $\ovi$. 
 \end{itemize} 
  In what follows we write $\vi_{ij}$ instead of $\vi_{i^+j^-}$.
  
 Note that $\bV$ has no direct summands of the form $\bS^-_j$. Indeed, such a direct summand gives a subspace $V'$
 of some $V^-_j$ such that $\ovi(V')=0$. It means that $\ker\vi\cap P^-(j)\not\sbe \gR P^-(j)$ which is impossible. If $M$
 is projective, $P^-=0$, so $\bV$ is a direct sum of simple modules of the form $\bS^+_i$. Vice versa,
 if $\bV$ has a direct summand of the form $\bS^+_i$, some space $V^+_i$ decomposes as $V'\+V''$, where
 $\im\ovi\cap V^+_i\sbe V''$. Then $P^+(i)=P'\+P''$, where $\im\vi\cap P^+(i)\sbe P''$, hence $M\simeq P'\+M'$. 
 Therefore, $M$ has no projective direct summands \iff $\Phi(M)\in\Ga\Mdd\bb$.
 
 On the contrary, let $\bV=\{V^\pm_i,\vi_{ij}\}$ be a representation of $\Ga$. 
 Set $P^\pm(i)=P_i^{\,d^\pm_i}$, 
 where $d^\pm_i=\dim_{F_i}V^\pm_i$, and $P^\pm=\bop_{i=1}^nP^\pm(i)$.
 Then $\vi_{ij}$ can be considered as a morphism $P^-(j)/\gR P^-(j)\to \gR P^+(i)$, so $\bV$ defines 
 a morphism $\vi:P^-\to P^+$ with $\im\vi\sbe\gR P^+$. If $\bV$ has no direct summands of the form $\bS^-_j$, 
 then $\ker\vi\sbe\gR P^-$, so we obtain a minimal projective presentation of the $A$-module $M=\cok\vi$.
 Obviously, $\Phi(M)\simeq \bV$ and $M$ has projective direct summands  \iff $\bV$ has direct summands isomorphic 
 to $\bS^+_i$. Thus $\Phi$ defines a bijection between isomorphism classes of $A$-modules
 without projective direct summands and $\Ga$-modules from $\Ga\Mdd\bb$.

 Let $\psi:Q^-\to Q^+$ be a minimal projective presentation of $N$ and $\Phi(N)=\{W^-_j,W^+_i,\psi_{ij}\}$. 
 If $f:M\to N$, it has a lift $(\al^+,\al^-)$ to a commutative diagram (of solid arrows)
 \begin{equation}\label{e11} 
  \vcenter{ \xymatrix{  P^- \ar[r]^\vi \ar[d]_{\al^-} & P^+ \ar[d]^{\al^+} \ar@{.>}[dl]_\xi \ar[r]^\pi & M \ar[d]^f \ar[r] & 0 \\
       					  Q^- \ar[r]^\psi 			& Q^+ \ar[r]^\th & N \ar[r] & 0.}} 
 \end{equation}
 Then $\al^\pm$ induce morphisms $\oal^\pm_{i}:V^\pm_i\to W^\pm_i$ which form a morphism of $\Ga$-modules
 $\oal=\Phi(f):\Phi(M)\to\Phi(N)$. If $f=0$, then $\al^+$ factors through $\psi$: $\al^+=\psi\xi$, whence 
 $\al^+(\gR P^+)=0$, so $\psi\al^-=0$ and $\im\al^-\sbe\ker\psi\sbe\gR Q^-$. Therefore, both $\oal^\pm=0$. It implies
 that $\oal$ does not depend on the lift of $\al$, so we can set $\oal=\Phi(f)$. On the contrary, each morphism 
 $\oal=\{\al_i^\pm\}:\bV\to\bW$ defines morphisms $\al^\pm$ which make the left square of the diagram \eqref{e11} 
 commutative, hence define a morphism $f:M\to N$ such that $\Phi(f)=\oal$. So we obtain a full functor 
 $\Phi:A\Mdd\to \Ga\Mdd$. 
 
 Note that if $P^+=U\+P'$, where $U$ is simple, $\im\vi=\ker\pi\sbe P'$, so $M\simeq U\+P'/\ker\pi$. Therefore, if $M$
 has no projective direct summands, $\soc P^+=\gR P^+$. Suppose that $f:M\to N$, where neither $M$ nor $N$ have 
 projective direct summands and $\oal=\Phi(f)=0$. Then, in particular, $\al^+(\gR P^+)=0$, hence 
 $\im\al^+\sbe \soc Q^+=\gR Q^+$. 
 As $\th$ maps $\gR Q^+$ onto $\gR N$, it implies that $\im f=\im\th\al^+\sbe \gR N$, hence $f(\gR M)=0$ and $f$ is actually 
 a morphism $M/\gR M\to \gR N$. As the modules $M/\gR M,\,\gR Q^+$ and $\gR N$ are semisimple and $\gR Q^+\to\gR N$ 
 is surjective, this morphisms factors through $\gR Q^+$, that is $f$ factors through the projective module $Q^+$. 
Therefore, the functor $\Phi$ induces a fully faithful functor $A\Mdp\to\Ga\Mdd\bb$, which we also denote by $\Phi$. 
As we have already known that it is dense, it is an equivalence of categories.
   \end{proof}
   
 We can also consider the category $A\Mdi=A\Mdd/\cI$, where $\cI$ is the ideal of morphisms that factor through
 injective modules, and \emph{minimal injective copresentation} of modules $M$, that is morphisms $\vi:E^-\to E^+$ of
 injective modules such that $M\simeq\ker\vi\spe\soc E^-$ and $\im\vi\spe\soc E^-$. Let $\lst En$ be representatives of
 isomorphism classes of indecomposable injective modules (they are injective envelopes of simple modules).
 Consider the bipartite species $\Ga_*$ such that $\ver\Ga_*^\pm=\{i_\pm\mid 1\le i\le n\}$, 
 $F_{i_\pm}=\End_AE_i/\rad\End_AE_i$ and $H(i_+j_-)=\hom_A(E_j/\soc E_j,\soc E_i)$. Suppose that $A$ is left artinian, 
 that is injective $A$-modules are direct sums of indecomposables \cite[Thm.\,3.46]{Lam}. Then, following the proof of 
 Theorem~\ref{t21}, one can prove the dual result.
  \begin{cthm}{2.1a}\label{t21a} 
  If the ring $A$ is left artinian, there is an equivalence $\Phi_*:A\Mdi\ito\Ga_*\Mdd\bb$.
  \end{cthm}

   \section{Irreducible morphisms}
 \label{s3}
 
  From now on we suppose that $A$ is an \emph{Artin algebra}, that is a finite algebra over a commutative artinian
  ring $K$. We also suppose that $A$ is \emph{connected} (does not split into a direct product of rings), hence 
  we can suppose that $K$
  is local. Let $\gM=\rad K$ and $\aK=K/\gM$. We denote by $A\mdd$ the category of finite (hence artinian) $A$-modules.
  All sets of morphisms $\hom_A(M,N)$, where $M,N\in A\mdd$, are finite $K$-modules. In particular, all endomorphism rings 
  $\End_AM$ are finite $K$-algebras. It implies that endomorphism rings of indecomposable modules are local, hence $A\mdd$
  is a local category, so its species $\Si(A\mdd)$ is defined. It is called the \emph{\AR\ species} $\ars(A)$ of the algebra $A$
  (the \emph{\AR\ quiver} of $A$ if this species is identified with a quiver, for instance, if $\aK$ is algebraically closed).
  All skew fields $F_M$ and all bimodules $\irr(N,M)$ are finite dimentional over $\aK$.
    
     Let $K^*$ denotes the injective envelope of the $K$-module $\aK$. For a $K$-module $M$ set $M^*=\hom_K(M,K^*)$. 
     If $M$ is a left (right) $A$-module, $M^*$ is a right (left) $A$-module and
     the functor $\rD:M\mps M^*$ defines an exact duality between the categories of left and of right finite 
     $A$-modules. Note that it is isomorphic to the functor $\hom_A(M,A^*)$ and we usually identify them. 
     We also define the functor $\vv$ between
    left and right finite modules setting $M\vv=\hom_A(M,A)$. Note that if $A$ is not self-injective, it is not an equivalence, 
    even not exact, though it establishes a duality between the categories of left and right finite projective modules over $A$. 
    Let $P^-\xarr\vi P^+$ be a minimal projective presentation of a non-projective module $M$. We call the \emph{transpose} of
    $M$ and denote by $\tr M$ the right module $\cok\vi\vv$.
    
    Here is a summary of the \AR\ theory which we will use further, see \cite{AR,ARS}.
    We denote $\cM=A\mdd$. A left (right) \emph{$\cM$-module} is, by definition, an additive covariant (contravariant) 
    functor from $\cM$ to the category of abelian groups. 
    
    \begin{thmd}[TD]
    Let $M$ and $N$ be indecomposable finite $A$-modules.
    \begin{enumerate}
    \item The $\cM$-modules $\Rad_A(N,-)$ and $\Rad_A(-,M)$ are finitely generated.
    \\[.5ex]
    {\rm A morphism $\al:N\to X$ (respectively, $\be:X\to M$) is said \emph{left almost split with the source $N$} 
    (respectively, \emph{right \as\ with the target $M$}) if it generates the $\cM$-module $\Rad_A(N,-)$ (respectively, 
    $\Rad_A(-,M)$) and $X$ has no non-trivial decomposition $X=X'\+X''$ such that $\im\al\sbe X'$ 
    (respectively, $\be(X'')=0$).}\!%
   \footnote{\,Actually, in the terms of \cite{ARS} it is the definition of \emph{minimal left (or right) \as\ morphism}, 
   but we never consider non-minimal ones.}
   \vskip.5ex
   \item A left (right) \as\ morphism with prescribed source (target) is unique up to an isomorphism.
   \item\label{td3}   For a morphism $\vi:N\to M$ the following conditions are equivalent:
   		\begin{enumerate}
   		\item $\vi$ is irreducible.
   		\item There is a left \as\ morphism $\al:N\to M\+M'$ such that $\vi=\pi\al$, where $\pi$ is the projection onto $M$.
   		\item There is a right \as\ morphism $\be:N\+N'\to M$ such that $\vi=\be\io$, where $\io$ is the embedding of $N$.
   		\end{enumerate}
   \item \label{td4}  If $M$ is projective, the embedding $\gR M\to M$ is right \as. 
   \item If $N$ is injective, the surjection $N\to N/\soc N$ is left \as.
   \item \label{td6}  For an exact sequence $0\to N\xarr\al X\xarr\be M\to 0$ the following conditions are equivalent:
   		\begin{enumerate}
   		\item  $\al$ is left \as.
   		\item  $\be$ is right \as.
   		\end{enumerate}
   	\vskip.5ex	
   {\rm If these conditions hold, this sequence is called an \emph{\ass\ with the source $N$ and the target $M$}.}
   \vskip.5ex
   \item An \ass\ with a fixed source or a fixed target is unique up to isomorphism.
   \item If $M$ is not projective, there is an \ass\ with the target $M$ and the source $\tau M=\rD\tr M$. \\[.5ex]
   {\rm $\tau M$ is called the \emph{\AR\ transform} of $M$.}
   \vskip.5ex
   \item  If $N$ is not injective, there is an \ass\ with the source $N$ and the target $\tau^{-1}N=\tr\rD N$. \\[.5ex]
   {\rm $\tau^{-1}N$ is called the \emph{inverse \AR\ transform} of $M$.}
    \end{enumerate}
   \rm Obviously, $\tau$ induces a one-to-one correspondence between isomorphism classes of indecomposable
   non-projective and non-injective modules and $\tau^{-1}$ is indeed its inverse.
    \end{thmd}
  
	From now on we suppose that $\gR^2=0$. In this case we can describe irreducible morphisms whose sources or
	targets are projective. We keep the notations of the previous section and denote by $U_i$	the simple module 
	$P_i/\gR P_i$. Let also $h_{ij}=\dim_{F_i}H_{ij}$.
  
   \begin{lem}\label{l21} 
   Let $U$ be a simple non-injective $A$-module. The irreducible morphisms with the source $U$ and indecomposable 
   targets are just all possible embeddings $U\to P_i$, where $P_i$ is not simple.
   \end{lem}
    \begin{proof}
    All of them are irreducible by TD(\ref{td4}) and TD(\ref{td3}b). Let $U\xarr\al M$ be irreducible and $M$ be indecomposable. 
    We can suppose that $\al$ is an embedding of a submodule $U\sb M$. Then $U\sbe\gR M$, otherwise $U$ is
    a direct summand of $M$. There is an epimorphism $\pi:P\to M$ which induces an epimorphism $\gR P\to\gR M$.
    Therefore, $\gR P$ contains a submodule isomorphic to $U$ and $\al$ factors through $\pi$. As $\al$ is irreducible,
    $\pi$ must split, that is $M$ is projective.
    \end{proof}
   
    \begin{prp}\label{p22} 
    If $P_i$ is not simple, $M$ is indecomposable and $\al:P_i\to M$ is irreducible, then $M$ is not projective.
    \end{prp}
     \begin{proof}
     Suppose $M$ is projective. Then $\al$ is not an epimorphism. As $\gR M$ is semisimple and is a unique maximal 
     submodule of $M$, $\im\al\simeq U_i\sbe\gR M$. Thus $\al$ factors as $P_i\to U_i\to M$, so is not irreducible.
     \end{proof}
         
 	 For every pair of indices $i,j$ choose a submodule of $\gR P_i$ which is a direct sum $\bop_{k=1}^{h_{ij}}T_{ijk}$,
     where $T_{ijk}\simeq U_j$ for all $k$. Define the morphisms $\th_{ijk}:U_j\to P_i$ such that $\im\th_{ijk}=T_{ijk}$.
     Note that $h_{ij}=0$ if $U_j$ is injective.
     
      \begin{lem}\label{l23}
     Let $U_j$ be non-injective, $\th_j:U_j\to \bop_{i=1}^nP_i^{h_{ij}}$ has the components $\th_{ijk}$ and 
     $\tU_j=\cok\th_j$. Then
     \begin{equation}\label{e21} 
      0\to U_j\xarr{\th_j} \bop_{i=1}^nP_i^{h_{ij}} \xarr{\eta_j} \tU_j \to 0  
     \end{equation}
     is an \ass. In particular, $\tU_j=\tau^{-1} U_j$.\\[.5ex]
     \emph{We  denote by $\eta_{ijk}:P_i\to\tU_j\ (1\le k\le  h_{ij})$ the components of $\eta_j$. }
      \end{lem}
       \begin{proof}
       All components $\th_{ijk}:U_j\to P_i$ of $\th_j$ are irreducible and every irreducible morphism $U_j\to P_i$ is a linear
       combination of components of $\th_{ijk}$, hence factors through $\th_j$. By Lemma~\ref{l21}, all irreducible morphisms
       with the source $U_j$ are components of $\th_j$. Therefore, the claim follows from TD(\ref{td3}) and TD(\ref{td6}).
       \end{proof}
       
     Note that $\th_{ijk}$ can be considered as a map $P_j/\gR P_j\to P_i$ and, if we identify them, they form a basis of
     $H_{ij}$. We denote by $\tH_{ji}$ the space with the basis $\{\eta_{ijk}\}$. 

    \begin{cor}\label{c24} 
    Every irreducible morphism with the source (with the target) $P_i$ is a right (respectively, left) linear combination of 
    $\eta_{ijk}$ (respectively, of $\th_{ijk}$) over $F_i$ for some fixed $j$. 
    \end{cor}
    
    Certainly, there are dual results concerning injective modules. We denote by $E_i$ the \emph{injective envelope} 
    of $U_i$ and $h'_{ji}=\dim_{F_j}\hom_A(E_i,U_j)$. Then $E_i/\soc E_i\simeq \bop_{j=1}^nU_j^{h'_{ji}}$.
     
     \begin{clem}{3.1a}\label{l31a} 
       Let $U$ be a simple non-projective $A$-module. The irreducible morphisms with the target $U$ and indecomposable 
       sources are just all possible surjections $E_j\to U$, where $E_i$ is not simple.
     \end{clem}
     
        \begin{cprp}{3.2a} \label{p32a} 
    If $E_i$ is not simple, $N$ is indecomposable and $\al:N\to E_i$ is irreducible, then $N$ is not injective.
    \end{cprp}
    
    	 For every pair of indices $i,j$ choose a submodule of $E_i/U_i$ which is a direct sum $\bop_{k=1}^{h'_{ji}}S_{jik}$,
     where $S_{jik}\simeq U_j$ for all $k$. Define the morphisms $\tau_{jik}:E_i\to U_j$ which are compositions 
     $E_i\to S_{jik}\ito U_j$. Note that $h'_{ji}=0$ if $U_j$ is projective.
   
       \begin{clem}{3.3a}\label{l23a} 
     Let $U_j$ be non-projective, $\tau_j: \bop_{i=1}^nE_i^{h'_{ji}}\to U_j$ has the components $\tau_{jik}$ and 
     $\tU'_j=\ker\tau_j$. Then
     \[
      0\to \tU'_j\xarr{\si_j} \bop_{i=1}^nE_i^{h_{ji}} \xarr{\tau_j} U_j \to 0 
     \]
     is an \ass. In particular, $\tU'_j=\tau U_j$.\\[.5ex]
     \emph{We denote by $\si_{jik}:\tU'_j\to E_i$ the components of $\si_j$. }
      \end{clem}
   
      \begin{ccor}{3.4a} \label{c34a} 
    Every irreducible morphism with the source (with the target) $E_i$ is a right (respectively, left) linear combination of 
    $\tau_{ijk}$ (respectively, of $\si_{ijk}$) over $F_i$ for some fixed $j$. 
    \end{ccor}
        
    As for irreducible morphisms whose sources and targets are not projective, we have the following result.
    
     \begin{lem}\label{l25} 
     Let $M,N$ be inecomposable non-projective modules. A morphism $\vi:N\to M$ is irreduciblle in $A\mdd$ \iff 
     $\Phi(\vi)$ is irreducible in $\Ga\mdd$. Therefore, $\Phi$ induces an isomorphism $\irr(M,N)\simeq\irr(\Phi(M),\Phi(N))$.
     \end{lem}
      \begin{proof}
      Let $\Phi(\vi)$ be non-irreducible that is $\Phi(\vi)=\sum_k u_kv_k$ for some non-invertible morphisms $u_k,v_k$.
      By Theorem~\ref{t21} we can suppose that $u_k$ and $v_k$ are from $\Ga\mdd\bb$, hence 
      $u_k=\Phi(u'_k),\,v_k=\Phi(v'_k)$ and $\vi=u'_kv'_k$ is also non-irreducible. On the other hand, let $\vi$ be non-irreducible,
      $\vi=\sum_k u_kv_k$ for some non-invertible $u_k$ and $v_k$. If all $u_kv_k$ factor through projective modules, 
      $\Phi(\vi)=0$, so is not irreducible. Otherwise it is the sum of such $\Phi(u_k)\Phi(v_k)$ that $u_kv_k$ do not factor through 
      projective modules, so is non-irreducible.      
      \end{proof}   
      
      For the species $\Ga^*$ we have an analogous result with analogous proof.
      
           \begin{clem}{3.5a}\label{l25a} 
     Let $M,N$ be inecomposable non-injective modules. A morphism $\vi:N\to M$ is irreduciblle in $A\mdd$ \iff 
     $\Phi_*(\vi)$ is irreducible in $\Ga_*\mdd$. Therefore, $\Phi$ induces an isomorphism 
     $\irr(M,N)\simeq\irr(\Phi_*(M),\Phi_*(N))$.
     \end{clem}

\section{Auslander-Reiten species}
 \label{s4} 

 Now we are able to restore the \AR\ species $\ars (A)$ starting from $\ars(\Ga)$. Note that the latter has been described 
 in \cite{Dlab}.
 
        First of all we consider projective and injective modules over the species 	$\Ga$ or, the same, over the algebra
   $A[\Si]$ from \eqref{eq0}. For each bimodule $H_{ij}=\hom_A(P_j,\gR P_i)$ we consider the dual bimodule
   $$H'_{ji}=\hom_\aK(H_{ij},\aK)\simeq \hom_{F_i}(H_{ij},F_i)\simeq \hom_{F_j}(H_{ij},F_j) $$
   (recall that our algebras are finite dimensional over $\aK$).
  We define $\Ga$-modules $P^-_j$ and $E^+_j$ as follows:
    \begin{align*}
     & P^-_j(i^-)= \begin{cases}
       F_j &\text{if } i=j,\\
       0 &\text{otherwise},
      \end{cases}\\
    & P^-_j(i^+)=H_{ij}\ \text{if}\ j\ne i.\\
    & E^+_j(i^+)= \begin{cases}
      F_j &\text{if } i=j,\\
      0 &\text{otherwise},
     \end{cases}\\
     & E^+_j(i^-)=H'_{ij}\ \text{if}\ j\ne i.
    \end{align*}
    
    One easily sees that $E^+_j\simeq \Phi(U_j)$ and $P^-_i\simeq \Phi(\tU_i)$.
 
    \begin{prp}\label{p41}  
     $P^-_j$ is an indecomposable projective $\Ga$-module and every indecomposable projective module is
     isomorphic either to $S^+_j$ or to $P^-_j$ for some $j$.
    \end{prp}
    
     \begin{cprp}{4.1a} \label{p41a}
    $E^+_j$ is an indecomposable injective $\Ga$-module and every indecomposable injective module is
     isomorphic either to $S^-_j$ or to $E^+_j$ for some $j$.
     \end{cprp}
     
	 The proofs of Propositions~\ref{p41} and \ref{p41a} follow immediately from the definitions and the description
      \eqref{eq0} of the algebra $\aK[\Ga]$.
 
  \begin{thm}\label{t42} 
   The \AR\ species $\ars(A)$ is obtained from the \AR\ species $\ars(\Ga)$ by the following procedure:
   \begin{enumerate}
   \item  Delete simple objects $\bS^+_i$ and $\bS^-_i$ and replace every non-simple object $\bV$ by $\Phi^{-1}(\bV)$.
   \item  Add objects $P_i$ with the corresponding skewfields $F_{P_i}=F_i$.
   \item  Define bimodules 
    \begin{align*}
     B(M,P_i)= \begin{cases}
       H_{ij} &\text{if } M=U_j,\\
       0 & \text{otherwise}.
      \end{cases}\\
     B(P_i,M)= \begin{cases}
       \tH_{ji} &\text{if } M=\tU_j,\\
       0 & \text{otherwise}.       
   `   \end{cases}
    \end{align*}
   \end{enumerate}
  \end{thm}
   \begin{proof}
    It follows immediately from Lemmata~\ref{l23} and \ref{l25}.
   \end{proof}
 
   \begin{exm}
    {\bf1.}~Let $A=\aK[x,y]/(x,y)^2$. Then $\Ga$ is the Kronecker quiver $\xymatrix{ 1^- \ar@/^/[r] \ar@/_/[r] & 1^+}$.
    $\ars(\Ga)$ consists of the preprojective component
   \[
     \xymatrix{  S^+_1  \ar@/^/[r] \ar@/_/[r] & P^-_1 \ar@/^/[r] \ar@/_/[r] &  X_1 \ar@/^/[r] \ar@/_/[r] & *+\txt{\dots}  },
   \]
   preinjective component
   \[
     \xymatrix{ *+{\dots} \ar@/^/[r] \ar@/_/[r] & Y_1 \ar@/^/[r] \ar@/_/[r] & E^+_1 \ar@/^/[r] \ar@/_/[r] & S^-_1      }
   \]
   and regular components which are homogeneous tubes. By Theorem~\ref{t42}, $\ars(A)$ consists
   of the same regular components and
   the component obtained by gluing the preinjective and the preprojective components with the projective
   module $A$:
   \[
     \xymatrix{ *+{\dots} \ar@/^/[r] \ar@/_/[r] & M \ar@/^/[r] \ar@/_/[r] & U \ar@/^/[r] \ar@/_/[r] & A 
       \ar@/^/[r] \ar@/_/[r] & \tU \ar@/^/[r] \ar@/_/[r] &  N \ar@/^/[r] \ar@/_/[r] & *+{\dots} }
   \]
   Note that $M\simeq A^*$ by Lemma~\ref{l31a} and $N\simeq\tau^{-1}A$ since $\tU=\tau^{-1}U$.
 
   \medskip
   {\bf2.}~Let now $A$ be given by the quiver
   \[
     \xymatrix{ 1 \ar[rr]|{\,a_1\,} \ar@/^1em/[rr]|{\,a_2\,} && 2 \ar@/^1em/[ll]|{\,b\,}  }     
   \]
    with relations $a_ib=ba_i=0\ (i=1,2)$. Then $\Ga$ consists of two components
    \[
      \xymatrix{ 1^- \ar[r] & 2^+} \ \text{ \rm and } \xymatrix{ 2^- \ar@/^/[r] \ar@/_/[r] & 1^+}.
    \]
  The second component is the Kronecker quiver, so the corresponding \AR\ components are as above,  just replace there
  $1^-$ by $2^-$. The \AR\ quiver of the first component is 
  \[
   S^+_2 \longrightarrow P_1^- \longrightarrow S_1^- \quad (P_1^-=E_2^+).
  \]
  Hence $\ars(A)$ consists of regular components, as those of the Kronecker quiver, and one component
  obtained by gluing the remaining three components of $\ars(\Ga)$ with the projective modules $P_1$ and $P_2$:
  {\footnotesize
  \[  
    \xymatrix{*+{\dots} \ar@/^/[r] \ar@/_/[r] & E_2 \ar@/^/[r] \ar@/_/[r] & U_1 \ar[r] & P_2 \ar[r] & U_2 \ar@/^/[r] \ar@/_/[r]
   & P_1 \ar@/^/[r] \ar@/_/[r] & \tU_2 \ar@/^/[r] \ar@/_/[r] & {\phantom{.}\tau^{-1}P_1} \ar@/^/[r] \ar@/_/[r] & *+{\dots} }
  \]}
  Note that $P_2\simeq E_1$ and $U_2\simeq \tU_1$.
   \end{exm}
   
   Certainly, one can obtain an analogous result using the quiver $\Ga_*$, Theorem~\ref{t21a} and 
   Lemmata~\ref{l23a} and \ref{l25a}. We leave the formulation of it to the reader.

\bibliographystyle{acm}
\bibliography{R^2.bib}

\end{document}